\documentclass[11pt,reqno]{amsart}
\usepackage{amssymb, amsmath, amsthm,amsrefs,marvosym,wasysym}
\usepackage{latexsym}
\usepackage{color}
\usepackage{exscale}
\usepackage{fullpage}
\usepackage{rotating}
\usepackage{bm, bbm, mathrsfs}
\usepackage{graphics, graphicx}
\newtheorem{theorem}{Theorem}[section]
\newtheorem{lemma}[theorem]{Lemma}
\newtheorem{proposition}[theorem]{Proposition}

\newtheorem{definition}{Definition}[section]

\newtheorem{remark}{Remark}[section]

\def\vp{\varphi}

\def\eq#1{(\ref{#1})}
\def\nn{\nonumber}
\def\({\left(\begin{array}{cccccc}}
\def\){\end{array}\right)}

\def\eq#1{(\ref{#1})}
\def\nn{\nonumber}

\def\({\left(\begin{array}{cccccc}}
\def\){\end{array}\right)}

\def\bes{\begin{eqnarray}}
\def\ees{\end{eqnarray}}

\newcommand{\del}{\partial}

\newcommand{\beq}{\begin{equation}}
\newcommand{\eeq}{\end{equation}}
\newcommand{\bea}{\begin{eqnarray}}
\newcommand{\eea}{\end{eqnarray}}
\newcommand{\beann}{\begin{eqnarray*}}
\newcommand{\eeann}{\end{eqnarray*}}

\newcommand{\RR}{\mathbb{R}}
\newcommand{\CC}{\mathbb{C}}

\newcommand{\bu}{{\bf u}}
\newcommand{\bx}{{\bf x}}

\DeclareMathOperator{\dv}{div}

\DeclareMathOperator\erf{erf}

\numberwithin{equation}{section}

\begin{document}

\title{On non-uniqueness for the system $\bu_t+(\bu\cdot\nabla)\bu=\mu\Delta{\bf u}$}
\author{Helge Kristian Jenssen }\thanks{H.~K.~Jenssen, Department of
Mathematics, Penn State University,
University Park, State College, PA 16802, USA ({\tt
jenssen@math.psu.edu}).}

\begin{abstract}
	Explicit irrotational solutions, obtained via the Cole-Hopf transform from the multi-d heat
	equation, give examples of non-uniqueness for the Cauchy problem in supercritical 
	$L^p$, $W^{1,p}$, and $W^{2,p}$ regimes.
	We verify non-uniqueness of the trivial solution in the sense of $L^p(\RR^n)$, whenever $n\geq2$
	and $1\leq p<n$. The same solutions give non-uniqueness in $W^{1,p}(\RR^n)$ 
	and $W^{2,p}(\RR^n)$ for $1\leq p<\frac{n}{2}$ and $1\leq p<\frac{n}{3}$, 
	respectively. The main example provides solutions which are classical  for 
	strictly positive times, and vanish in the stated norms, but explode in $L^\infty(\RR^n)$,
	as $t\to0+$. The non-uniqueness is unrelated to the 
	Tikhonov non-uniqueness phenomenon for the heat equation. 

\bigskip 

\end{abstract}

\maketitle

\section{Introduction}\label{rad_ch}
\subsection{The Cole system and the Cole-Hopf transform}
We consider the Cauchy problem for the $n$-dimensional semilinear parabolic system 
\beq\label{c}
	\bu_t+(\bu\cdot\nabla)\bu=\mu\Delta{\bf u},
\eeq
where 
\[\bu=(u_1,\dots,u_n)^T:[0,\infty)\times\RR^n_{\bf x}\to\RR^{n\times1}_\bu.\]
Here $(\bu\cdot \nabla)\bu\equiv (D_\bx\bu) \bu$, where $D_\bx\bu=$ 
Jacobian matrix, and $\Delta$ acts component-wise; note that $\bu$ and $\bx$ 
are both $n$-dimensional. The constant $\mu>0$ is fixed throughout.

The system \eq{c} is usually called the vector Burgers equation or the $n$-dimensional 
Burgers system. We will refer to it as the {\em Cole system} and its solutions as {\em Cole solutions}: As far as 
we known, \eq{c} was first considered by Cole \cite{cole}. 
He observed that irrotational solutions, i.e., $\del_{x_i}u_j\equiv\del_{x_j}u_i$, 
can be obtained via the Cole-Hopf transform
\beq\label{ch}
	{\bu}(t,{\bf x})=-\frac{2\mu}{\theta(t,{\bf x})}\nabla\theta(t,{\bf x}),
\eeq
where $\theta(t,\bx)$ is a solution of the scalar $n$-dimensional heat equation
\beq\label{heat}
	\theta_t=\mu\Delta \theta.
\eeq
Note that ${\bu}$ in \eq{ch} is in general singular at points where $\theta$ vanishes.

In this note we give explicit examples of radial solutions to \eq{c} of the form 
\eq{ch} which attain vanishing initial data in the sense of $L^p$ whenever 
$n\geq2$ and $1\leq p<n$. The same solutions give non-uniqueness also in 
Sobolev spaces; e.g., in
$W^{1,p}(\RR^n)$ and $W^{2,p}(\RR^n)$ for $1\leq p<\frac{n}{2}$ and 
$1\leq p<\frac{n}{3}$, respectively. 

The Cole system \eq{c} enjoys a parabolic scaling invariance: 
If $\bu(t,\bx)$ is a solution to \eq{c}, then so is $\bu_\lambda(t,\bx):=\lambda \bu(\lambda^2 t,\lambda\bx)$, 
for any $\lambda>0$. We have 
\[\|\bu_\lambda(t)\|_{L^p(\RR^n)}=\lambda^{1-\frac{n}{p}}\|\bu(t)\|_{L^p(\RR^n)}\qquad\text{and}\qquad
\|\bu_\lambda(t)\|_{\dot W^{k,p}(\RR^n)}=\lambda^{k+1-\frac{n}{p}}\|\bu(t)\|_{\dot W^{k,p}(\RR^n)},\]
so that $p=n$ and $p=\frac{n}{k+1}$, are  $L^p$- and $W^{k,p}$-critical for \eq{c}, respectively.
Ill-posedness in the supercritical regimes $1\leq p<n$ and $1\leq p<\frac{n}{k+1}$ is therefore not surprising,
cf.\ Section 3.1 in \cite{tao}. 

Our main points are to provide simple concrete examples of ill-posedness
in the form of non-uniqueness, and to verify that this non-uniqueness 
is not ``caused'' by non-uniqueness for the heat equation (Tikhonov phenomenon; see 
Remark \ref{tikh}).

\subsection{Setup and results}
We consider radial, and thus irrotational, solutions to \eq{c} of the form
\beq\label{rad}
	{\bf u}(t,{\bf x})=u(t,r)\hat{\bf r},\qquad \hat{\bf r}
	=\textstyle\frac{{\bf x}}{r},\quad r=|{\bf x}|.
\eeq
The primary goal is to demonstrate, by concrete examples of the form \eq{rad}, that there is
an infinity of non-trivial solutions $\bu\not\equiv\bf 0$ to \eq{c} which attain vanishing initial data in 
the sense of $L^p$ whenever $n\geq2$ 
and $1\leq p<n$. 

The corresponding radial heat functions $\theta(t,r)$ will be smooth and strictly positive
solutions of \eq{heat} on $(0,\infty)\times \RR$. In the main example 
(Section \ref{Lp_non-uniq}), these generate classical and everywhere-defined 
solutions $\bu(t,\bx)$ to \eq{c} on $(0,\infty)\times \RR^n$ via \eq{rad}. 
While these Cole solutions satisfy $\bu(t)\to0$ as $t\to0+$ in $L^p(\RR^n)$ with $1\leq p<n$
(Lemma \ref{prop_1}), they suffer blowup in $L^\infty(\RR^n)$ as $t\to0+$.

Two further cases of non-uniqueness are obtained by considering 
radial self-similar and radial stationary solutions (Section \ref{rad_ss_statn}). However, these examples exhibit 
singularities either in their initial data or along $\{t>0,\bx=\bf0\}$.

One source of interest in the system \eq{c} is its formal likeness to 
the Navier-Stokes and (for $\mu=0$) Euler systems, in which case 
$\bu$ in \eq{c} is thought of as a velocity field of a fluid. For a radial field \eq{rad} it is then 
natural to require that $\bu(t,{\bf 0})\equiv{\bf 0}$ due to symmetry.
This is indeed the case in our main example of non-uniqueness in  Section \ref{Lp_non-uniq}. 
However, the fluid analogy should not be taken too literally: The 
absence  of a pressure term in \eq{c} leaves its predictive power for Navier-Stokes doubtful;
see Remark \ref{sing_form}.

A general result for symmetrizable parabolic systems guarantees local existence and uniqueness
of $H^s(\RR^n)$-solutions whenever $s>\frac{n}{2}+1$; see \cites{vh,kaw,ser,tay}.
(The works \cites{vh,kaw,ser} consider a wider class of symmetrizable 
hyperbolic-parabolic systems, not necessarily in divergence form, which
includes the compressible Navier-Stokes and MHD systems.) 
We note that  \eq{c} may be written in the form
\beq\label{non_div_syst}
	\bu_t+\sum_{i=1}^n{\bf A}_i(\bu)\bu_{x_i}=\mu\Delta \bu,\qquad {\bf A}_i(\bu)=u_i{\bf I}_{n}
\eeq
(${\bf I}_n=n\times n$ identity matrix), 
so that the Cole system is a concrete case of a multi-d symmetric semilinear and 
uniformly parabolic system.
As is clear from above, the observed cases of non-uniqueness occur in spaces
that are much larger than $H^s(\RR^n)$ with $s>\frac{n}{2}+1$. On the other hand, we are not aware 
of examples or general results that would narrow the gap.

Unsurprisingly, all examples of non-uniqueness that we consider requires $n\geq2$, i.e., 
they apply to multi-d systems of equations. Unless explicitly stated otherwise, 
$n\geq2$ is assumed in what follows.

\subsection{Radial Cole solutions}
A calculation shows that if a solution ${\bf u}$ to \eq{c} 
is of the radial form \eq{rad}, then $u(t,r)$ solves the $n$-dimensional 
radial Cole equation\footnote{Note that \eq{c_rad} may be 
written in divergence form, viz.\ 
$u_t+(\frac{1}{2}u^2)_r=\mu\big(u_r+\textstyle\frac{n-1}{r}u\big)_r$.
More generally,  within the 
class of irrotational solutions, \eq{c} may be re-written in divergence form as
$\bu_t+\dv(\frac{1}{2}|\bu|^2{\bf I}_n)=\mu\Delta\bu$.}
\beq\label{c_rad}
	u_t+uu_r=\mu\big(u_r+\textstyle\frac{n-1}{r}u\big)_r,
	\qquad u:\RR^+_t\times\RR^+_r\to\RR.
\eeq
For radial solutions the Cole-Hopf transform \eq{ch} gives 
that any solution $\theta(t,r)$ of the $n$-dimensional radial heat equation
\beq\label{rad_heat}
	\theta_t=\mu\big(\theta_{rr}+\textstyle\frac{n-1}{r}\theta_r\big),
\eeq
yields a  solution 
\beq\label{ch_rad}
	u=-2\mu\textstyle\frac{\theta_r}{\theta}
\eeq
to \eq{c_rad}, and thus a solution ${\bf u}(t,{\bf x})$ to \eq{c} via \eq{rad}.

The last statement is a little vague as we have not specified in what
sense the resulting ${\bf u}(t,{\bf x})$ should solve \eq{c}. There are 
several facets to this. First, as noted, Cole solutions of the form \eq{ch} 
may suffer singularities where $u(t,r)$ blows up in amplitude
due to vanishing of $\theta(t,r)$; for concrete 1-d examples, see \cite{bp}
and Remarks \ref{tikh}-\ref{sing_form} below.
Next, the Cole-system \eq{c} in dimensions $n\geq2$ is not in divergence form, 
and therefore does not admit an obvious notion of weak solutions (this issue is 
carefully discussed  in \cite{jlz}). 
Finally, for the particular case of radial solutions, there is the issue of 
what type of behavior one should allow as $r\to0+$ at positive times. 

In the {\em ad hoc} Definition \ref{acceptable} we delimit the class of solutions we are primarily 
interested in: An ``$L^p$-acceptable'' solution of a Cauchy problem is required to 
solve \eq{c} in a pointwise classical sense for $t>0$, while attaining the initial 
data in $L^p$. In particular, a radial solution $\bu(t,\bx)$ of this type must necessarily 
vanish along $\RR^+\times \{\bf0\}$. 
As the initial data are required to be attained only in the sense of $L^p$, blowup in 
amplitude as $t\to0+$ is allowed.

We note that the heat functions (i.e., solutions of \eq{heat}) we make use 
of are elementary, and the corresponding solutions to \eq{c} are known in the literature,
for both the 1-d and multi-d cases \cites{bp,light,js}. 
Our main example is analyzed in Section \ref{Lp_non-uniq} where we detail the 
$L^p$-estimates (for $1\leq p<n$) of the Cole solution given by \eq{soln_ex1}-\eq{u_def}. 
The result is summarized in Proposition \ref{main}. 

More singular examples, which do not satisfy Definition \ref{acceptable}, 
are analyzed briefly in Section \ref{rad_ss_statn}. In particular, we obtain $L^p$
non-uniqueness (for $n=3$) when the initial datum is a certain stationary 
profile for \eq{c}. General stationary solutions have recently 
been analyzed in detail in \cite{jlz}, where various stability results for such 
solutions are established. Our findings are not at variance with these results:
The instability we observe takes place in spaces that are larger than the ones 
considered in \cite{jlz}.

Before proceeding we comment briefly on how the non-uniqueness
comes about, explaining why it is not simply non-uniqueness for \eq{heat}
``translated'' to \eq{c} via Cole-Hopf.

For our main example we start from two distinct initial conditions for the $n$-dimensional 
heat equation \eq{heat}, viz.\ $\theta(0)\equiv a$ and $\theta(0)= a+\delta_{\bf 0}$, where $a>0$
is a constant.
These data generate two distinct heat functions $\theta\equiv a$ and $\theta=a+G_n$, 
where $G_n$ denotes the $n$-dimensional heat kernel. In turn, \eq{ch} yield two distinct Cole 
solutions which, however, agree at time zero in the sense of $L^p(\RR^n)$ whenever $1\leq p<n$. 
The Cole solution corresponding to $\theta\equiv a$
is the trivial solution $\bu\equiv \bf 0$; the one corresponding to $\theta=a+G_n$ is given in \eq{soln_ex1}-\eq{u_def} below;
a particular case is displayed in Figure \ref{Figure_1}.

\begin{remark}[Non-unique Cole solutions and the Tikhonov phenomenon]\label{tikh}
	It is well-known that the Cauchy problem for the heat equation admits multiple 
	solutions unless growth conditions are imposed (Tikhonov phenomenon, 
	see, e.g.\ Section 5.5.1 in \cite{dib}). This is therefore a possible source, via 
	the Cole-hopf transform, of non-uniqueness for the Cole system \eq{c}. 
	However, as outlined above, the two Cole solutions we consider are obtained from 
	distinct initial data for the heat equation \eq{heat}. Therefore, the particular type
	of non-uniqueness we exhibit for \eq{c} is unrelated to the Tikhonov phenomenon.

	Of course, if $\theta_1,\theta_2$ are distinct heat functions with identical initial data, 
	then these will generate, via the Cole-Hopf transform, distinct solutions $\bu_1,\bu_2$ to \eq{c} with identical 
	initial data. However, according to Widder's theorem (p.\ 134 in \cite{wid})\footnote{Widder's 
	formulation in \cite{wid} applies to 1-d non-negative heat functions continuous on $[0,\infty)\times\RR$;
	later extensions by Aronson and others cover more general parabolic equations 
	and allows for measures in the initial data, see \cites{aron}.}, 
	one of $\theta_1,\theta_2$ must change sign for $t>0$. It follows that one of $\bu_1,\bu_2$ must contain 
	singularities and fail to be everywhere defined on $(0,\infty)\times \RR^n$. The latter scenario is
	not pursued in the present work.
\end{remark}

\begin{remark}[Singularity formation]\label{sing_form}
	The work \cite{ps} shows how, for {\em complex-valued} solutions $u\in\CC$, 
	the scalar Burgers equation (i.e., \eq{c} with $n=1$) can develop $L^\infty$-blowup 
	from $C^\infty_c$-data. 
	(The corresponding $2\times2$-system for the real 
	and imaginary parts of $u$ is not of parabolic or  hyperbolic-parabolic type.)
	The maximum principle for \eq{heat} prevents such behavior for real-valued and 
	irrotational solutions of \eq{c}. 
	
	Concerning amplitude blowup, the recent works \cites{mrrs1,mrrs2} 
	have shown how radially focusing wave solutions to the compressible Navier-Stokes
	system can generate amplitude blowup in finite time. The solutions are perturbations 
	of self-similar focusing Euler flows.
	It is instructive to see how the maximum principle prevents such behavior 
	in radial Cole solutions $\bu=u\hat{\bf r}$. As $u$ itself gives the convective speed of propagation, 
	the idea would be to have a big, negative ``down-bump'' in $u$ drive itself toward $r=0$.
	However, the minimum value in such a solution cannot decrease: 
	At the point of minimum we have $u_r=0$, so that \eq{c_rad} gives 
	$u_t=\mu u_{rr}-\frac{(n-1)u}{r^2}>0$ there, and blowup is averted.
	
	This provides a concrete example of how the Cole system fails to capture 
	a fundamental property of the Navier-Stokes system. 
	(It is still conceivable that the Cole system can admit singularity 
	formation in non-radial solutions.)
\end{remark}


%


{\bf Notation:} We set $\RR^+:=(0,\infty)$ and $\RR^+_0:=[0,\infty)$.
For real-valued functions $f$ of the time variable $t$ 
and the spatial variable $\bx\in\RR^n$ or $r\in\RR^+$, the 
time argument is always listed first. E.g., if $f$ is defined for positive times, then 
$f:\RR^+_t\times\RR^n_\bx\to\RR$ or $f:\RR^+_t\times\RR_r\to\RR$ (and similarly for
vector valued functions).
For $\bu:\RR^+_0\times\RR^n\to\RR^n$ and $p\geq 1$, let 
\[\|\bu(t)\|_p=\Big(\int_{\RR^n}|\bu (t,\bx)|^p\, d\bx\Big)^\frac{1}{p},\]
where $|\cdot|$ denotes the Euclidean norm on $\RR^n$. We write
$A\lesssim B$ to mean that there is a finite constant $C>0$ so that $A\leq CB$.

%


\section{Non-uniqueness in $L^p$}\label{Lp_non-uniq}
In this section we consider radial Cole solutions $\bu(t,\bx)$ that solve \eq{c} in a classical sense
for $t>0$, and which attain their initial data in the sense of $L^p$.
For lack of a better terminology, we refer to such solutions as ``$L^p$-acceptable.''
(This notion is relaxed in Section \ref{rad_ss} when we consider stationary 
and self-similar solutions.)

\begin{definition}[$L^p$-acceptable solutions]\label{acceptable}
	Given a measurable function $\bu_0:\RR^n\to\RR^n$ and $p\geq1$.
	We say that $\bu:\RR^+\times\RR^n\to\RR^n$ is an {\em $L^p$-acceptable solution} 
	to the Cauchy problem for the Cole system {\em \eq{c}} with initial data $\bu_0$, provided:
	\begin{enumerate}
		\item[(i)] $\bu_t,D_\bx\bu,\Delta\bu\in C(\RR^+\times\RR^n)$;
		\item[(ii)] $\bu$ satisfies \em{\eq{c}}  classically at all points $(t,\bx)\in\RR^+\times\RR^n$;
		\item[(iii)] $\bu(t)\to\bu_0$ in $L^p(\RR^n)$ as $t\to0+$.
	\end{enumerate}
\end{definition} 
\begin{remark}\label{rmk_defn}
	No further condition is imposed on the initial data $\bu_0$ beyond measurability; in particular,
	$\bu_0$ could be unbounded and/or not an element of $L^p(\RR^n)$. (For our main example
	$\bu_0\equiv \bf0$, while an unbounded $\bu_0$ is considered in Section \ref{rad_ss_statn}.)
	
         We note that any $L^p$-acceptable solution $\bu$ of the  radial form \eq{rad} 
         must necessarily vanish along $\RR^+\times\{{\bf0}\}$: 
         $\bu=u\frac{\bx}{|\bx|}$ itself is continuous at all points $(t,{\bf 0})$ ($t>0$) according to (i), and 
         the discontinuity of $\bx\mapsto\frac{\bx}{|\bx|}$
         at the origin therefore requires that $\bu(t,{\bf 0})={\bf 0}$ for all $t>0$.
\end{remark}

It is immediate that ${\bf u}\equiv{\bf 0}$ is an $L^p$-acceptable 
radial solution with data ${\bf u}_0\equiv{\bf 0}$, for any $p\geq1$.
We proceed to show that it is not the unique $L^p$-acceptable solution with vanishing data,
whenever $1\leq p<n$. 
For this, consider the solution 
\beq\label{heat_soln}
	\theta(t,r)=a+G_n(t,r)
\eeq
to \eq{rad_heat}, where $a>0$ is a constant and $G_n(t,r)$ denotes the $n$-dimensional 
heat kernel
\beq\label{heat_kernel_n}
	G_n(t,r)=\textstyle\frac{1}{(4\pi\mu t)^\frac{n}{2}}\exp(-\frac{r^2}{4\mu t}).
\eeq
As $\theta$ solves \eq{rad_heat}, the Cole-Hopf 
transform yields the positive solution
\beq\label{soln_ex1}
	\bar u(t,r)=\frac{r}{t\left[1+a(4\pi\mu t)^\frac{n}{2}\exp(\frac{r^2}{4\mu t})\right]}
\eeq
to the radial Cole equation \eq{c_rad} on $\RR^+\times\RR^+$. 
We note that, while
\[\lim_{{r\to\bar r}\atop{t\to0+}}\bar u(t,r)=0 \qquad\text{for $\bar r>0$ fixed,}\]
we have, e.g., $\bar u(t,\sqrt{t})\to\infty$ as $t\to0+$. The convergence to vanishing 
initial data is therefore pointwise but not uniform.
A plot of the solution \eq{soln_ex1} with $n=3$, $\mu=0.1$ and 
$a=1$ is given in Figure \ref{Figure_1};  evidently, it provides an example of  
a spherically expanding viscous wave which slows down and decays as it 
moves outward.
\begin{figure}
	\centering
	\includegraphics[width=9cm,height=9cm]{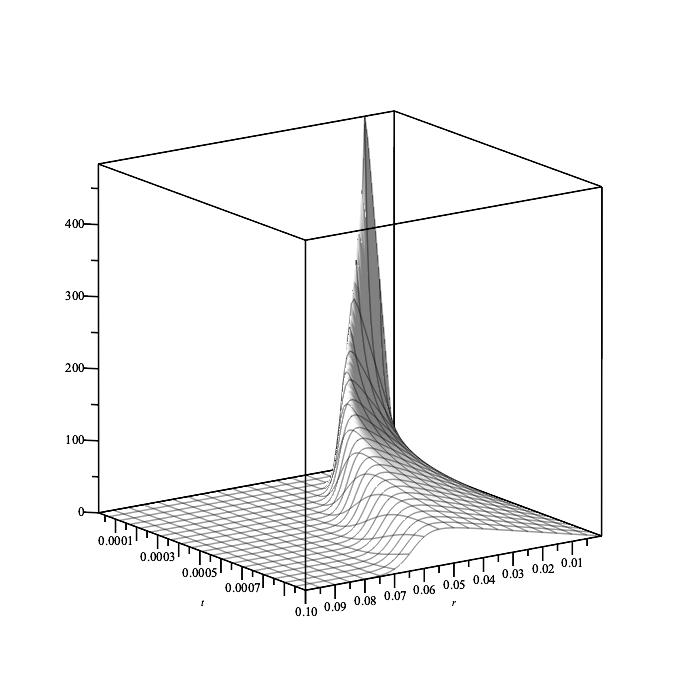}
	\caption{Maple plot of the solution \eq{soln_ex1} in the case $n=3$, with $\mu=0.1$ and $a=1$.
	The plot is for $0.0001\leq r\leq 0.1$ and $0.00002\leq t\leq 0.001$.}\label{Figure_1}
\end{figure} 

We now define $\bar \bu:\RR^+\times\RR^n\to\RR^n$ by setting
\beq\label{u_def}
	\bar \bu(t,\bx):=\left\{
	\begin{array}{ll}
		\bar u(t,|\bx|)\textstyle\frac{\bx}{|\bx|} & \text{for $t>0,|\bx|>0$}\\\\
		{\bf 0} & \text{for $t>0,\bx={\bf 0}$.}
	\end{array}\right.
\eeq 
The claim is that $\bar \bu$ is an $L^p$-acceptable solution to \eq{c} with initial data $\bar \bu_0\equiv0$ 
according to Definition \ref{acceptable}.

By construction we have that $\bar \bu(t,\bx)$ is smooth and satisfies \eq{c} 
classically on $\RR^+\times (\RR^n\smallsetminus\{{\bf0}\})$.
It remains to verify that $\bar \bu_t, D\bar \bu,\Delta\bar \bu$ are continuous at points along $\RR^+\times\{{\bf0}\}$,
that  \eq{c} is classically satisfied there, and that condition (iii) in Definition \ref{acceptable}
is met. We start with the last issue, for which it is convenient to
formulate two auxiliary lemmata.
\begin{lemma}\label{lem_1}
	For constants $n\geq1,k>-1, b>0, q>0$, $\ell>0$ let
	\[I(t):=t^q\int_0^\infty\frac{s^k}{(1+bt^\frac{n}{2}\exp(s))^\ell}\, ds.\]
	Then
	\[\lim_{t\to0+} I(t)=0.\]
\end{lemma}
\begin{remark}
	The limit is of the indeterminate form ``$\, 0\cdot\infty$.''
\end{remark}
\begin{proof}
	With $A(t):=bt^\frac{n}{2}$ and for times $t>0$ so small that 
	\[\log\textstyle \frac{1}{A(t)}>0,\]
	we write
	\[I(t)=t^q\int_0^{\log \frac{1}{A(t)}} \frac{s^k}{(1+A(t)\exp(s))^\ell}\, ds
	+t^q\int_{\log \frac{1}{A(t)}}^\infty \frac{s^k}{(1+A(t)\exp(s))^\ell}\, ds=:I_1(t)+I_2(t).\]
	With the stated assumptions we have
	\[I_1(t)< t^q\int_0^{\log \frac{1}{A(t)}} s^k\, ds=\frac{t^q}{k+1}
	\left(\log \textstyle\frac{1}{A(t)}\right)^{k+1}
	\propto t^q(-\log b-\textstyle\frac{n}{2}\log t)^{k+1}\to 0,\]
	as $t\to0+$. 
	To estimate $I_2(t)$ we first fix $\beta$ with $0<\beta<\min(\frac{2q}{n},\ell)$. 
	There is then an $s^*$ (depending on $n,k,q,\ell$) so that
	\[s^k\leq \exp(\beta s)\qquad \text{whenever $s>s^*$.}\]
	For times $t>0$ so small that 
	\[\log\textstyle \frac{1}{A(t)}>s^*,\]
	we therefore have (changing integration variable to $\xi=s-\log \frac{1}{A(t)}$)
	\begin{align}
		I_2(t)&=t^q\int_{\log \frac{1}{A(t)}}^\infty \frac{s^k}{(1+A(t)\exp(s))^\ell}\, ds\nn\\
		&\leq t^q\int_{\log \frac{1}{A(t)}}^\infty \frac{\exp(\beta s)}{(1+A(t)\exp(s))^\ell}\, ds\nn\\
		&\propto t^{q-\frac{\beta n}{2}}\int_0^\infty\frac{\exp(\beta \xi)}{(1+\exp(\xi))^\ell}\, d\xi
		\leq t^{q-\frac{\beta n}{2}}\int_0^\infty e^{(\beta-\ell) \xi}\, d\xi.
	\end{align}
	By the choice of $\beta$, the last integral is finite and $q>\frac{\beta n}{2}$, so that 
	$I_2(t)\to 0$ as $t\to0+$.
\end{proof}

\begin{lemma}\label{lem_2}
	For constants $n\geq1$, $b>0$, $c>-1$, $\ell> 0$, and $d>-\frac{c+1}{2}$,
	let
	\[J(t):=t^d\int_0^\infty\frac{r^c}{(1+bt^\frac{n}{2}\exp(\frac{r^2}{4\mu t}))^\ell}\, dr,\]
	Then
	\[\lim_{t\to0+} J(t)=0.\]
\end{lemma}
\begin{proof}
	Using the variable $s=\frac{r^2}{4\mu t}$ we have
	\[J(t)\propto t^{d+\frac{c+1}{2}}\int_0^\infty \frac{s^\frac{c-1}{2}}{(1+bt^\frac{n}{2}\exp(s))^\ell}\, ds,\]
	and the claim follows by applying Lemma \ref{lem_1} with $q=d+\frac{c+1}{2}$ and $k=\frac{c-1}{2}$.
\end{proof}

\begin{lemma}\label{prop_1}
	The function $\bar {\bf u}(t,{\bf x})$ defined in \eq{soln_ex1}-\eq{u_def} 
	satisfies condition {\em (iii)} in Definition \ref{acceptable} with initial data 
	$\bar \bu_0\equiv{\bf 0}$ whenever $n\geq 2$ and $1\leq p<n$.
\end{lemma}
\begin{proof}	
The claim is that
\[\int_0^\infty |\bar u(t,r)|^pr^{n-1}\, dr\to0
		\qquad\text{as $t\to0+$.}\]
Substituting from \eq{soln_ex1} we have 
\beq\label{key_int}
	\int_0^\infty |\bar u(t,r)|^pr^{n-1}\, dr
	=t^{-p}\int_0^\infty \frac{r^{p+n-1}\, dr}{\big(1+a(4\pi\mu t)^\frac{n}{2}
	\exp(\frac{r^2}{4\mu t})\big)^p}.
\eeq
Applying Lemma \ref{lem_2} with $b=a(4\pi\mu)^\frac{n}{2}$, $d=-p$, $c=p+n-1$, and
$\ell=p$, 
and observing that the restrictions on the parameters in Lemma \ref{lem_2} 
are met because of the assumption $1\leq p<n$, we obtain the claim.
\end{proof}

Next, consider the behavior of $\bar \bu(t,\bx)$ along $\RR^+\times\{{\bf0}\}$.
We fix  $\bar t>0$ and employ the following notations:
\[r:=|\bx|,\qquad \xi:=\textstyle\frac{r^2}{4\mu t}, \qquad A:=(4\pi\mu t)^\frac{n}{2},\qquad \bar A:=(4\pi\mu \bar t)^\frac{n}{2}.\]
Then, for $t>0$ and $\bx\neq0$, the radial part and the $i$th component of $\bar \bu(t,\bx)$ are
\beq\label{uu_i}
	\bar u(t,r)=\textstyle\frac{r}{t(1+Ae^\xi)}, \qquad\text{and}\qquad 
	\bar u_i(t,\bx)=\textstyle\frac{\bar u(t,r)}{r}x_i,
\eeq
respectively. In particular, $\bu(t,\bx)$ is continuous along $\RR^+\times\{{\bf0}\}$, i.e.,
\beq\label{cont}
	\lim_{{\bx\to{\bf 0}}\atop{t\to\bar t}}\bar \bu(t,\bx)={\bf0}\qquad\text{for $\bar t>0$.}
\eeq
Direct calculations using \eq{uu_i}${}_1$ give
\beq\label{fracs}
	\textstyle\frac{\bar u(t,r)}{r}=\textstyle\frac{1}{t(1+Ae^\xi)},\qquad
	\big(\textstyle\frac{\bar u(t,r)}{r}\big)_r=-\textstyle\frac{Are^\xi}{2\mu t^2(1+Ae^\xi)^2},
	\qquad\text{and}\qquad 
	\textstyle\frac{1}{r}\big[\frac{1}{r}\big(\textstyle\frac{\bar u(t,r)}{r}\big)_r\big]_r
	=\textstyle\frac{Ae^\xi(Ae^\xi-1)}{4\mu^2t^3(1+Ae^\xi)^3},
\eeq
so that
\beq\label{1st_id}
	\lim_{{\bx\to{\bf 0}}\atop{t\to\bar t}}{\textstyle\frac{\bar u(t,r)}{r}=\textstyle\frac{1}{\bar t}\frac{1}{1+\bar A}},\qquad
	\lim_{{\bx\to{\bf 0}}\atop{t\to\bar t}}\big({\textstyle\frac{\bar u(t,r)}{r}}\big)_r=0,
	\qquad\text{and}\qquad 
	\lim_{{\bx\to{\bf 0}}\atop{t\to\bar t}}\textstyle\frac{1}{r}\big[\frac{1}{r}\big(\textstyle\frac{\bar u(t,r)}{r}\big)_r\big]_r
	=\textstyle\frac{\bar A(\bar A-1)}{4\mu^2\bar t^3(1+\bar A)^3}.
\eeq
For later reference we note that, since $r\mapsto\frac{\bar u(t,r)}{r}$ is smooth and even 
(when regarded as a function of $r\in\RR$ for $t>0$ fixed), \eq{fracs}${}_{1,2}$ and \eq{1st_id}${}_1$ give
\beq\label{useful}
	\big({\textstyle\frac{\bar u(\bar t,r)}{r}}\big)_r=O(r)\qquad\text{and}\qquad
	\big|\textstyle\frac{\bar u(\bar t,r)}{r}-\frac{1}{\bar t}\frac{1}{1+\bar A}\big|=O(r^2)
	\qquad\text{as $r\to0$.}
\eeq
With $\del_j\equiv\del_{x_j}$ we have from \eq{uu_i}${}_2$, with $t>0$ and $\bx\neq0$, that
\beq\label{1st_spatials_a}
	\del_j\bar u_i(t,\bx)=\big(\textstyle\frac{\bar u(t,r)}{r}\big)_r\frac{x_ix_j}{r}+\frac{\bar u(t,r)}{r}\delta_{ij},
\eeq
and \eq{1st_id}${}_{1,2}$ give
\beq\label{1st_spatials_b}
	\lim_{{\bx\to0}\atop{t\to\bar t}}\del_j\bar u_i(t,\bx)=\textstyle\frac{1}{\bar t}\frac{1}{1+\bar A}\delta_{ij}.
\eeq
Recalling that $\bar u_i(\bar t,{\bf 0})=0$,  \eq{uu_i}${}_2$ gives ($\hat{\bf e}_j=$ $j$th unit coordinate vector)
\[\del_j\bar u_i(\bar t,{\bf 0})=\lim_{h\to0}{\textstyle\frac{1}{h}\bar u_i(\bar t,h\hat{\bf e}_j)}
=\lim_{h\to0}{\textstyle\frac{\bar u(\bar t,|h|)}{|h|}\delta_{ij}}
=\textstyle\frac{1}{\bar t}\frac{1}{1+\bar A}\delta_{ij}.\]
Comparing this with \eq{1st_spatials_b} shows that $D_\bx\bar \bu$ exists 
and is continuous at all points of $\RR^+\times\{{\bf0}\}$.

Next, applying  $\del_{k}$ to \eq{1st_spatials_a} gives, for $t>0$ and $\bx\neq0$, that
\beq\label{2nd_spatials_a}
	\del_{jk}\bar u_i(t,\bx)=\textstyle\frac{1}{r}\big[\frac{1}{r}\big(\textstyle\frac{\bar u(t,r)}{r}\big)_r\big]_rx_ix_jx_k
	+\frac{1}{r}\big(\textstyle\frac{\bar u(t,r)}{r}\big)_r\big[x_i\delta_{jk}+x_j\delta_{ik}+x_k\delta_{ij}\big].
\eeq
According to \eq{1st_id}${}_{2,3}$ we therefore have
\beq\label{2nd_spatials_b}
	\lim_{{\bx\to {\bf 0}}\atop{t\to\bar t}}\del_{jk}\bar u_i(t,\bx)=0.
\eeq
To calculate $\del_{jk}\bar u_i(\bar t,{\bf 0})$ we first note that, by \eq{1st_spatials_a}, we have
\[\del_j\bar u_i(\bar t,h\hat{\bf e}_k)-\del_j\bar u_i(\bar t,{\bf 0})
=\big(\textstyle\frac{\bar u}{r}\big)_r\big|_{(\bar t,|h|)}\cdot|h|\delta_{ik}\delta_{jk}
+\big(\textstyle\frac{\bar u}{r}\big|_{(\bar t,|h|)}-\textstyle\frac{1}{\bar t(1+\bar A)}\big)\delta_{ij},\]
so that \eq{fracs}${}_1$, \eq{1st_id}${}_2$, and \eq{useful} give
\[\del_{jk}\bar u_i(\bar t,{\bf 0})
:=\lim_{h\to0}\textstyle\frac{1}{h}\big[\del_j\bar u_i(\bar t,h\hat{\bf e}_k)-\del_j\bar u_i(\bar t,{\bf 0})\big]=0.\]
Comparing with \eq{2nd_spatials_b} we conclude that all second partials $\del_{jk}\bar \bu$
exist, vanish, and are continuous at all points of $\RR^+\times\{{\bf0}\}$.

For the $t$-derivative of $\bar \bu$, it is immediate from \eq{u_def} that
$\bar \bu_t(\bar t,{\bf 0})={\bf 0}$. Finally, since $\bar \bu$ by construction satisfies \eq{c} 
at points away from $\RR^+\times\{{\bf0}\}$, the calculations above give
\[\lim_{{\bx\to{\bf 0}}\atop{t\to\bar t}}\bar \bu_t(t,\bx)
=\lim_{{\bx\to{\bf 0}}\atop{t\to\bar t}}\big[\Delta\bar \bu(t,\bx)-(D_\bx\bar \bu(t,\bx))\bar \bu(t,\bx)\big]={\bf 0},\]
where the last equality follows from \eq{cont}, \eq{1st_spatials_b}, and \eq{2nd_spatials_b}.
Thus, $\bar \bu_t$ exists, vanishes, and is continuous at all points of $\RR^+\times\{{\bf0}\}$.
In particular, \eq{c} is classically satisfied along $\RR^+\times\{{\bf0}\}$, and therefore
on all of $\RR^+\times\RR^n$.

This concludes the argument for showing that the vector field $\bar \bu(t,\bx)$ defined in 
\eq{soln_ex1}-\eq{u_def} is an $L^p$-acceptable solution of \eq{c} with vanishing initial data.
Recalling that $a>0$ is an arbitrary constant in the argument above, we have
proved:
\begin{proposition}\label{main}
	Consider the Cole system {\em \eq{c}} in dimension $n\geq 2$ and assume $1\leq p<n$.
	Then there are infinitely many non-trivial $L^p$-acceptable solutions 
	to \eq{c} with vanishing initial data.
\end{proposition}

\section{Self-similar and stationary solutions}\label{rad_ss_statn}
In Definition \ref{acceptable} we required that an ``$L^p$-acceptable'' solution
should solve \eq{c} classically at all points of $\RR^+\times\RR^n$. In particular,
a radial $L^p$-acceptable solution must vanish identically along $\RR^+\times\{{\bf 0}\}$.
In contrast, the recent work \cite{jlz} introduces several definitions of weak solutions
of \eq{c} which allow for singularities (e.g., infinite amplitudes) to be present.
(We comment further on the findings of \cite{jlz} in Section \ref{hi_reg} when 
we consider higher regularity estimates.)

In this section we consider two further examples of $L^p$-non-uniqueness in this latter setting
where one of the solutions blows up along $\RR^+\times\{{\bf 0}\}$. The first of these come about by searching for  
self-similar solutions. In this case we obtain examples much like that in Section \ref{Lp_non-uniq}: 
There is an infinity of self-similar and non-vanishing solutions 
$\bu_\text{ss}(t,\bx)$, all of which tend to $\bf 0$ in $L^p(\RR^n)$ as $t\to0+$ 
(again for $1\leq p<n$). On the other hand, these solutions contain a singularity along $\RR^+\times\{{\bf 0}\}$.

For the second example we restrict to $n=3$ for convenience, and search for a
non-unique continuation of singular initial data. 
A simple instance occurs by prescribing the  stationary profile $\bu_\text{st}(\bx)=\frac{2\mu}{|\bx|^2}\bx$ as initial data. 
The situation is now in a sense opposite to that of the self-similar solutions described above:
We identify a single non-stationary solution, which is a classical solution to 
\eq{c} on $\RR^+\times\RR^n$,
and which tends to the singular stationary solution $\bu_\text{st}(\bx)$
in $L^p(\RR^n)$ ($1\leq p<n$) as $t\to0+$.

Since these examples of non-uniqueness involves singular solution (in particular, they are not 
$L^p$-acceptable according to Definition \ref{acceptable}), we regard them as  
somewhat weaker than the solution $\bar\bu(t,\bx)$ considered in Section \ref{Lp_non-uniq}.

\subsection{$L^p$-non-uniqueness via self-similar solutions}\label{rad_ss}
The scaling invariance for \eq{c} noted earlier is inherited by the 
radial Cole equation \eq{c_rad}: If $u(t,r)$ solves \eq{c_rad}, then so does 
\[u_\lambda(t,r):=\lambda u(\lambda^2 t,\lambda r)\qquad\text{for any $\lambda>0$.}\]
A {\em self-similar solution} $u(t,r)$ is required to satisfy  $u\equiv u_\lambda$ for all $\lambda>0$.
It follows that a self-similar solution must  be of the form
\beq\label{xi_version}
	u(t,r)=\textstyle\frac{1}{\sqrt{t}}\vp(\eta)\qquad\qquad \eta:=\textstyle\frac{r}{\sqrt{t}}.
\eeq
Solutions to the radial Cole equation \eq{c_rad} of this type
may be obtained by applying the Cole-Hopf transform \eq{ch}
to solutions of the form $\theta(t,r)=f(\eta)$ of the $n$-dimensional 
radial heat equation \eq{rad_heat}:
\beq\label{u_ss_ch}
	u_{\text{ss}}(t,r)=-\textstyle\frac{2\mu}{\sqrt{t}}\frac{f'(\eta)}{f(\eta)}.
\eeq

We proceed to identify a class of self-similar radial solutions to \eq{c_rad} by 
determining the general form of $f$ in \eq{u_ss_ch}. Substitution of $\theta(t,r)=f(\eta)$ in 
\eq{rad_heat} yields the ODE
\[f''=-(\textstyle\frac{\eta}{2\mu}+\frac{n-1}{\eta})f'\qquad('\equiv\frac{d}{d\eta}),\]
which can be integrated twice to give
\beq\label{f'}
	f'(\eta)=K\eta^{1-n}e^{-\frac{\eta^2}{4\mu}}
\eeq
and (changing integration variable to $s=\frac{\eta^2}{4\mu}$)
\beq\label{f}
	f(\eta)=A-K\int_\eta^\infty z^{1-n}e^{-\frac{z^2}{4\mu}}\, dz
	=A-{\textstyle\frac{K}{2}}(4\mu)^{1-\frac{n}{2}}\int_\frac{\eta^2}{4\mu}^\infty s^{-\frac{n}{2}}e^{-s}\, ds,
\eeq
where $A$ and $K$ are arbitrary constants of integration. Substituting into \eq{u_ss_ch}
and setting 
\[a:=-\textstyle\frac{2A}{K(4\mu)^{1-\frac{n}{2}}},\]
yield the solution
\beq\label{u_ss}
	u_{\text{ss}}(t,r)=(4\mu)^\frac{n}{2}t^{\frac{n}{2}-1}r^{1-n}\cdot
	\frac{e^{-\frac{r^2}{4\mu t}}}{a+\int_\frac{r^2}{4\mu t}^\infty s^{-\frac{n}{2}}e^{-s}\, ds}, 
\eeq
to the radial Cole equation \eq{c_rad}. Here $a$ is an arbitrary constant which we 
choose positive in order that $u_{\text{ss}}(t,r)$ takes bounded and positive values whenever $r,t>0$.

However, differently from the solution $\bar u$ in \eq{soln_ex1}, $u_{\text{ss}}(t,r)$ (for $n\geq 2$) 
blows up like $\sim\frac{1}{r}$ as $r\to0+$ with $t>0$ fixed; this follows from \eq{u_ss} and L'H\^opital's rule.
(Note that $a>0$ requires $A$ and $K$ to have opposite signs, so that the heat function $\theta(t,r)=f(\eta)$
in \eq{f} does not change sign. Thus, the singularity in $u_{\text{ss}}(t,r)$ along $r=0$ is not due to a change 
of sign in the underlying heat function.)
A plot of the solution \eq{u_ss} with $n=3$, $\mu=0.005$ and 
$a=1$ is given in Figure \ref{Figure_2}.

Just as for the solutions in \eq{soln_ex1}, we have
\[\lim_{{r\to\bar r}\atop{t\to0+}}u_{\text{ss}}(t,r)=0 \qquad\text{for $\bar r>0$ fixed,}\]
and $u_{\text{ss}}(t,\sqrt{t})\to\infty$ as $t\to0+$. Thus, the convergence to 
vanishing initial data is again pointwise but not uniform.

\begin{figure}
	\centering
	\includegraphics[width=9cm,height=9cm]{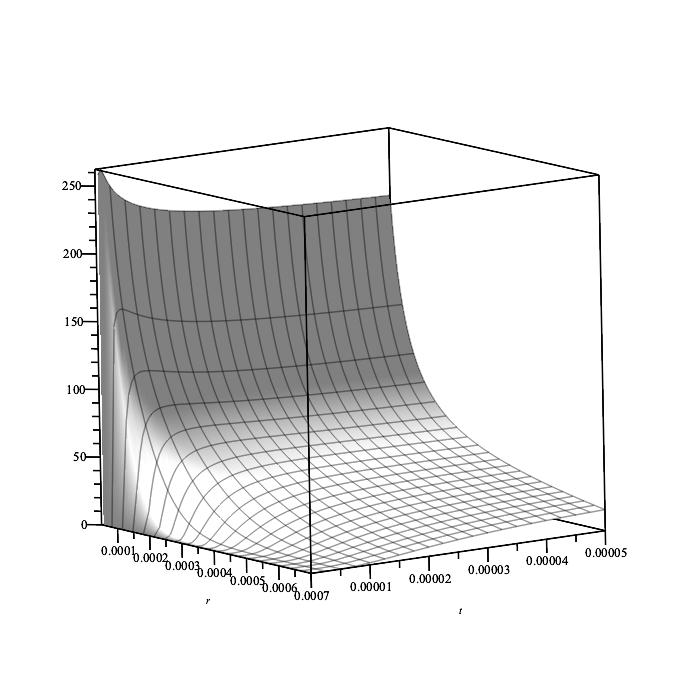}
	\caption{Maple plot of the solution \eq{u_ss} in the case $n=3$, with $\mu=0.005$ and $a=1$.
	The plot is for $0.00005\leq r\leq 0.0007$ and $0\leq t\leq 0.00005$.}\label{Figure_2}
\end{figure} 

Defining the corresponding vector field
\[\bu_{\text{ss}}(t,\bx):=u_{\text{ss}}(t,|\bx|)\textstyle\frac{\bx}{|\bx|} \qquad\text{for $t>0,|\bx|>0$}\]
(and $\bu_{\text{ss}}(t,\bf 0)\equiv\bf 0$, say), we proceed to show that $\bu_{\text{ss}}$
attains vanishing initial data in the sense of $L^p$ whenever $1\leq p<n$.
Changing to integration variable $z=\frac{r^2}{4\mu t}$ we get
\beq\label{expr}
	\|\bu_{\text{ss}}(t)\|_p^p\propto\int_0^\infty|u_{\text{ss}}(t,r)|^pr^{n-1}\, dr 
	\propto t^{\frac{n-p}{2}}\int_0^\infty 
	\frac{z^\frac{(1-n)(p-1)-1}{2}e^{-pz}}{\big[a+\int_z^\infty s^{-\frac{n}{2}}e^{-s}\, ds\big]^p}\, dz.
\eeq
It is clear that the integrand in the last expression in \eq{expr} is integrable at $z=\infty$, and a
calculation using L'H\^opital's rule shows that it is of order $z^{\frac{n-p}{2}-1}$ for $z\approx 0+$.
The integral is therefore finite whenever $1\leq p<n$, and \eq{expr} gives
$\|\bu_{\text{ss}}(t)\|_p\to0$ as $t\to0+$ for any choice of of the constant $a>0$. 


We conclude: Self-similar solutions provide further examples of $L^p$-non-uniqueness
in the Cauchy problem for the $n$-dimensional Cole system whenever $1\leq p<n$.

However, due to their singular behavior along $\RR^+\times\{\bf 0\}$, the self-similar 
solutions $\bu_\text{ss}(t,\bx)$ are not $L^p$-acceptable in the sense of Definition 
\ref{acceptable}.

\subsection{$L^p$-non-uniqueness of stationary solutions in 3-d}\label{statn}
In this section we consider a different situation for the case $n=3$.
The goal is to give an example of $L^p$-non-uniqueness for a particular stationary 
Cole solution $\bu_\text{st}(\bx)$.
The latter will be singular along $\RR^+\times\{\bf 0\}$, and is therefore
not itself an $L^p$-acceptable solution according to Definition 
\ref{acceptable}. On the other hand, we shall identify a non-stationary and $L^p$-acceptable 
solution to the Cauchy problem with initial data $\bu_\text{st}$ (in the sense of $L^p$ for $1\leq p<3$).

We start by listing the stationary solutions of the radial heat equation \eq{rad_heat} 
in dimensions $n\geq 2$:
\[\theta(r)=
	\left\{
	\begin{array}{ll}
		C_1\log r+C_2 & \text{for $n=2$}\\\\
		C_1r^{2-n}+C_2& \text{for $n\geq 3$,}
	\end{array}
	\right.
\]
where $C_1,C_2$ are arbitrary constants.
Applying the Cole-Hopf transform \eq{ch_rad} we obtain the corresponding 
non-trivial and singular stationary solutions to the radial Cole equation \eq{c_rad}:
\beq\label{statn_c_soln}
	u(r)=
	\left\{
	\begin{array}{ll}
		-\textstyle\frac{2\mu}{r(\log r+C)} & \text{for $n=2$}\\\\
		\textstyle\frac{2(n-2)\mu}{r(1+Cr^{n-2})}& \text{for $n\geq 3$,}
	\end{array}
	\right.
\eeq
where $C$ is an arbitrary constant.
\begin{remark}
	The solutions in \eq{statn_c_soln} do not exhaust the family of 
	stationary radial Cole solutions. For a complete list in the case
	$n=3$, see Theorem 2.1 in \cite{jlz}.
\end{remark}
Specializing to $n=3$ we single out the 3-d stationary solution with $C=0$ in \eq{statn_c_soln}, i.e.,
\[u_\text{st}(r):=\textstyle\frac{2\mu}{r},\]
and define
\beq\label{u_st_def}
	\bu_\text{st}(\bx):=u_\text{st}(|\bx|)\textstyle\frac{\bx}{|\bx|}\qquad \text{for $\bx\neq\bf 0$},
\eeq
and $\bu_\text{st}(\bf0)\equiv\bf0$, say.
To identify a non-stationary Cole solution taking $\bu_\text{st}$ as initial data in the sense of 
$L^p$ (this solution will be denoted $\bu_\text{nst}$ below), we exploit that the 3-d radial heat equation 
is reducible to the 1-d heat equation.
Specifically, to solve \eq{rad_heat} with $n=3$ it is
convenient to set
\beq\label{Theta_theta}
	\Theta(t,r):=r\theta(t,r),
\eeq
which yields the 1-d heat equation for $\Theta$: $\Theta_t=\mu\Theta_{rr}$ for $r>0$.
With $\Theta$ determined, \eq{Theta_theta} and \eq{ch_rad} 
give the following expression for the corresponding solution $u(t,r)$ to the 3-d radial Cole 
equation:
\beq\label{u_Theta}
	u(t,r)=2\mu\big(\textstyle\frac{1}{r}-\textstyle\frac{\Theta_r(t,r)}{\Theta(t,r)}\big)
	\qquad\qquad (n=3),
\eeq
($u(t,r)$ is therefore the sum of the 3-d stationary solution $u_\text{st}$
and the Cole-Hopf transform of a solution to the 1-d heat equation, 
i.e., a solution to Burgers equation $b_t+bb_r=\mu b_{rr}$.)

We are free to assign boundary conditions along $r=0$ for $\Theta$. 
In order that the corresponding radial Cole-solution $\bu(t,\bx)$ be $L^p$-acceptable,
we need $u$ to vanish along $r=0$ (cf. Remark \ref{rmk_defn}).
We should therefore choose $\Theta$ to vanish along $r=0$ as well.
To see this, at least at a formal level, note that $\Theta(t,0)\equiv 0$ yields 
$\Theta_{rr}(t,0)=\mu^{-1}\Theta_t(t,0)\equiv 0$,
so that $\Theta(t,r)=r\Theta_r(t,0)+O(r^3)$, and a further Taylor expansion about $r=0$ gives
\[u(t,r)=2\mu\big(\textstyle\frac{1}{r}-\textstyle\frac{\Theta_r(t,r)}{\Theta(t,r)}\big)
=O(r)\qquad\text{as $r\downarrow0$.}\]
As a concrete example, consider the 1-d heat function $\Theta$ obtained 
by imposing $\Theta(t,0)\equiv 0$ together with the (non-compatible) initial data 
$\Theta(0,r)\equiv1$. The solution in this case is 
\[\Theta(t,r)=\erf(\textstyle\frac{r}{\sqrt{4\mu t}}),\]
where $\erf$ denotes the error function, 
\[\erf(z)={\textstyle\frac{2}{\sqrt\pi}}\int_0^z e^{-s^2}\, ds.\]
According to \eq{u_Theta} the corresponding non-stationary solution to the radial Cole 
equation \eq{c_rad} is given by
\beq\label{u_nst}
	u_\text{nst}(t,r)=2\mu\Big(\textstyle\frac{1}{r}-
	\textstyle\frac{\exp(-\frac{r^2}{4\mu t})}{\sqrt{\mu\pi t}\erf(\textstyle\frac{r}{\sqrt{4\mu t}})}\Big).
\eeq
A concrete plot is shown in Figure \ref{Figure_3}; as in the case of the solution in \eq{soln_ex1},
$u_\text{nst}$ is a decaying and expanding viscous wave.
\begin{figure}
	\centering
	\includegraphics[width=9cm,height=9cm]{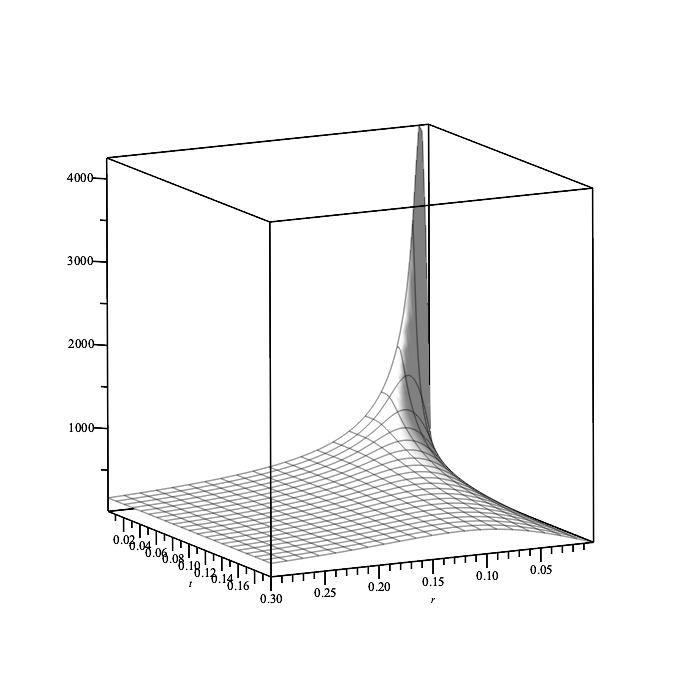}
	\caption{Maple plot of the solution \eq{u_nst} to the 3-d Cole equation \eq{c_rad} 
	with $n=3$ and $\mu=0.01$.
	The plot is for $0.001\leq r\leq 0.3$ and $0.001\leq t\leq 0.2$.}\label{Figure_3}
\end{figure} 

We define the corresponding non-stationary solution $\bu_\text{nst}:\RR^+\times\RR^3\to\RR^3$ 
to the Cole system \eq{c}  by 
\beq\label{u_nst_def}
	\bu_\text{nst}(t,\bx):=\left\{
	\begin{array}{ll}
		u_\text{nst}(t,|\bx|)\textstyle\frac{\bx}{|\bx|} & \text{for $t>0,|\bx|>0$}\\\\
		{\bf 0} & \text{for $t>0,\bx={\bf 0}$.}
	\end{array}\right.
\eeq 
By construction, $\bu_\text{nst}$ is smooth and satisfies \eq{c} 
classically on $\RR^+\times (\RR^3\smallsetminus\{{\bf0}\})$.
By Taylor expanding $\erf$ and $\exp$ about zero, it is straightforward to verify that 
\[\lim_{{r\to 0+}\atop{t\to\bar t}} u_\text{nst}(t,r)=0\qquad\text{for any $\bar t>0$,}\]
so that, according to \eq{u_nst_def}, $\bu_\text{nst}$ is at least continuous on $\RR^+\times\RR^3$.

Next, for $\bar r>0$ fixed, \eq{u_nst} gives
\[\lim_{{r\to \bar r}\atop{t\to0+}}u_\text{nst}(t,r)= \textstyle\frac{2\mu}{\bar r},\]
so that $\bu_\text{nst}(t)\to\bu_\text{st}$ in a pointwise manner as $t\to0+$.
On the other hand, evaluating along $r=\sqrt{t}$ gives
\[|u_\text{nst}(t,\sqrt{t})-u_\text{ss}(\sqrt{t})|=\textstyle\frac{C_\mu}{\sqrt t},\]
where $C_\mu>0$, so that the convergence $u_\text{nst}(t)\to u_\text{ss}$ 
as $t\to0+$ is pointwise but not uniform. 

We proceed to verify that $\bu_\text{nst}(t)\to\bu_\text{st}$ in $L^p(\RR^3)$ whenever 
$1\leq p<3$. We have, upon changing integration variable to $s=\frac{r}{\sqrt{4\mu t}}$,
\begin{align}
	\|\bu_\text{nst}(t)-\bu_\text{st}\|_p^p
	&\propto \int_0^\infty|u_\text{nst}(t,r)-u_\text{st}(r)|^pr^2\, dr\nn\\
	&\propto t^{-\frac{p}{2}}\int_0^\infty 
	\textstyle\frac{\exp(-p\frac{r^2}{4\mu t})r^2}{\big(\erf(\frac{r}{\sqrt4\mu t})\big)^p}\, dr\nn\\
	& \propto t^{\frac{3-p}{2}}\int_0^\infty 
	\textstyle\frac{\exp(-ps^2)s^2}{(\erf(s))^p}\, ds.\label{last}
\end{align}
The integrand in the last expression is clearly integrable at $s=+\infty$; as $\erf(s)=O(s)$
as $s\to0+$, it is integrable at $s=0+$ provided $s^{2-p}$ is integrable, i.e., provided $p<3$,
in which case \eq{last} gives that $\|\bu_\text{nst}(t)-\bu_\text{st}\|_{L^p(\RR^3)}\to 0$
as $t\to0+$.

Finally, calculations similar to those in Section \ref{Lp_non-uniq} (omitted) show that $\bu_\text{nst}$
is an $L^p$-acceptable solution to the Cauchy problem for \eq{c} with initial data $\bu_\text{st}$
according to Definition \ref{acceptable}.

We conclude: For $n=3$ the solution $\bu_\text{nst}$, defined in \eq{u_nst}-\eq{u_nst_def}, 
provides an example of $L^p$-non-uniqueness in the Cauchy problem for the $3$-dimensional 
Cole system whenever $1\leq p<3$.

Again, however, this example is somewhat weaker than the one in
Section \ref{Lp_non-uniq} since the initial data is a stationary solution which
is not  $L^p$-acceptable in the sense of Definition \ref{acceptable}.

\section{Higher regularity}\label{hi_reg}
In this section we briefly consider higher regularity estimates of the solutions 
$\bar\bu$, $\bu_\text{ss}$, $\bu_\text{nst}$ constructed above.
As noted, the Cole system \eq{c} is an example of a 
symmetric, semilinear, and uniformly parabolic system, for which 
the Cauchy problem is uniquely solvable in $H^s(\RR^n)$ whenever
$s>1+\frac{n}{2}$ \cites{vh,kaw,ser,tay}. The estimates we establish below for 
the concrete solution $\bar\bu(t,\bx)$ show that it is rather far away 
from this regularity class. The same applies for the regularity class 
considered in \cite{jlz}, where perturbations of stationary 
solutions $\phi(r)$, behaving like $\sim\frac 1 r$ near $r=0$, are measured 
in $W^{1,\infty}(\RR^n)$. Thus, there is no contradiction between these results 
and the non-uniqueness results established in the present work.

We restrict attention to the solution $\bar\bu$ from Section \ref{Lp_non-uniq} and 
estimate its $W^{1,p}(\RR^n)$- and $W^{2,p}(\RR^n)$-distances to the 
trivial solution $\bu\equiv\bf 0$. We shall show that $\bar\bu$ yields 
non-uniqueness in these spaces provided that $1\leq p<\frac n 2$ and $1\leq p<\frac n 3$,
respectively. In particular, for $p=2$, non-uniqueness requires spatial
dimensions $n\geq 5$ and $n\geq 7$, respectively. 

Recalling the result in Lemma \ref{prop_1}, the claims above amount to: $\|\del_{j}\bar u_i(t)\|_p\to0$
as $t\to0+$ provided $1\leq p<\frac n 2$, and  $\|\del_{jk}\bar u_i(t)\|_p\to0$ 
as $t\to0+$ provided $1\leq p<\frac n 3$, for all indices $1\leq i,j,k\leq n$.
To argue for this, we first note that the general formulae \eq{1st_spatials_a} and \eq{2nd_spatials_a} give  
\[|\del_j\bar u_i(t,\bx)|\lesssim |\bar u_r|+ |\textstyle\frac{\bar u}{r}|\]
and
\[|\del_{jk}\bar u_i(t,\bx)|\lesssim |\bar u_{rr}|
+ |\textstyle\frac{\bar u_r}{r}|+ |\textstyle\frac{\bar u}{r^2}|,\]
where $\bar u$ and its derivatives are evaluated at $(t,r)$.
With $\bar u$ given by \eq{soln_ex1}, and setting 
\[f(t,r)=1+a(4\pi\mu t)^\frac{n}{2}\exp(\textstyle\frac{r^2}{4\mu t}),\]
we have 
\[\bar u(t,r)= \textstyle\frac{r}{tf(t,r)},\]
and direct calculations give, in turn,
\[|\textstyle\frac{\bar u}{r}|\lesssim \textstyle\frac{t^{-1}}{f(t,r)},\]
\[|\bar u_r|\lesssim |\textstyle\frac{\bar u}{r}|+\textstyle\frac{t^{-2}r^2}{f(t,r)}
\lesssim \textstyle\frac{t^{-1}}{f(t,r)}+\textstyle\frac{t^{-2}r^2}{f(t,r)},\]
and
\[|\bar u_{rr}|\lesssim |\textstyle\frac{\bar u_r}{r}|+|\textstyle\frac{\bar u}{r^2}|
+\textstyle\frac{t^{-2}r}{f(t,r)}+\textstyle\frac{t^{-3}r^3}{f(t,r)}
\lesssim \textstyle\frac{t^{-1}r^{-1}}{f(t,r)}+\textstyle\frac{t^{-2}r}{f(t,r)}
+\textstyle\frac{t^{-3}r^3}{f(t,r)},\]
where we have repeatedly used that $t^\frac{n}{2}\exp(\frac{r^2}{4\mu t})\lesssim f(t,r)$.
Making use of these expressions we get,
\begin{align}
	\|\del_j\bar u_i(t)\|_p^p&\lesssim\int_0^\infty \big(|\bar u_r|^p+ |\textstyle\frac{\bar u}{r}|^p\big)r^{n-1}\,dr\nn\\
	&\lesssim t^{-p}\int_0^\infty {\textstyle\frac{r^{n-1}}{f(t,r)^p}}\, dr + t^{-2p}\int_0^\infty \textstyle\frac{r^{2p+n-1}}{f(t,r)^p}\, dr,
	\label{W1p}
\end{align}
and
\begin{align}
	\|\del_{jk}\bar u_i(t)\|_p^p&\lesssim\int_0^\infty \big(|\bar u_{rr}|^p+ |\textstyle\frac{\bar u_r}{r}|^p
	+ |\textstyle\frac{\bar u}{r^2}|^p\big)r^{n-1}\,dr\nn\\
	&\lesssim t^{-p}\int_0^\infty {\textstyle\frac{r^{n-p-1}}{f(t,r)^p}}\, dr + t^{-2p}\int_0^\infty{\textstyle\frac{r^{p+n-1}}{f(t,r)^p}}\, dr
	+ t^{-3p}\int_0^\infty \textstyle\frac{r^{3p+n-1}}{f(t,r)^p}\, dr.
	\label{W2p}
\end{align}
Applying Lemma \ref{lem_2} to each of the terms in \eq{W1p} and \eq{W2p}, we obtain the claims above, viz.
\[\|\del_j\bar u_i(t)\|_p\to 0\qquad\text{as $t\to0+$ provided $1\leq p<\textstyle\frac{n}{2}$,}\]
and
\[\|\del_{jk}\bar u_i(t)\|_p\to 0\qquad\text{as $t\to0+$ provided $1\leq p<\textstyle\frac{n}{3}$.}\]

\section*{Declarations}
\subsection*{Funding} No funding was received for conducting this study.

\subsection*{Conflict of interest} The author states that there is no conflict of interest.

\subsection*{Data availability} All data are available within the manuscript.

\section*{Acknowledgements}
The author is indebted to Anna Mazzucato for several helpful discussions.


\begin{bibdiv}
\begin{biblist}
\bib{aron}{article}{
   author={Aronson, D. G.},
   title={Non-negative solutions of linear parabolic equations},
   journal={Ann. Scuola Norm. Sup. Pisa Cl. Sci. (3)},
   volume={22},
   date={1968},
   pages={607--694},
   issn={0391-173X},
   review={\MR{0435594}},
}\bib{bp}{article}{
   author={Benton, Edward R.},
   author={Platzman, George W.},
   title={A table of solutions of the one-dimensional Burgers equation},
   journal={Quart. Appl. Math.},
   volume={30},
   date={1972},
   pages={195--212},
   issn={0033-569X},
   review={\MR{0306736}},
   doi={10.1090/qam/306736},
}
\bib{cole}{article}{
   author={Cole, Julian D.},
   title={On a quasi-linear parabolic equation occurring in aerodynamics},
   journal={Quart. Appl. Math.},
   volume={9},
   date={1951},
   pages={225--236},
   issn={0033-569X},
   review={\MR{0042889}},
   doi={10.1090/qam/42889},
}
\bib{dib}{book}{
   author={DiBenedetto, Emmanuele},
   title={Partial differential equations},
   series={Cornerstones},
   edition={2},
   publisher={Birkh\"auser Boston, Ltd., Boston, MA},
   date={2010},
   pages={xx+389},
   isbn={978-0-8176-4551-9},
   review={\MR{2566733}},
   doi={10.1007/978-0-8176-4552-6},
}
\bib{jlz}{article}{
   author={Ji, Shanming},
   author={Li, Zongguang},
   author={Zhu, Changjiang},
   title={Removable singularities and unbounded asymptotic profiles of
   multi-dimensional Burgers equations},
   journal={Math. Ann.},
   volume={391},
   date={2025},
   number={1},
   pages={113--162},
   issn={0025-5831},
   review={\MR{4846779}},
   doi={10.1007/s00208-024-02917-6},
}
\bib{js}{article}{
   author={Joseph, K. T.},
   author={Sachdev, P. L.},
   title={Initial boundary value problems for scalar and vector Burgers
   equations},
   journal={Stud. Appl. Math.},
   volume={106},
   date={2001},
   number={4},
   pages={481--505},
   issn={0022-2526},
   review={\MR{1825846}},
   doi={10.1111/1467-9590.00175},
}
\bib{kaw}{book}{
   author={Kawashima, Shuichi},
   title={Systems of a hyperbolic-parabolic composite type, with applications
	to the equations of magnetohydrodynamics},
   series={Doctoral Thesis, Kyoto University},
   date={1983},
   pages={iv+200},
}
\bib{light}{article}{
   author={Lighthill, M. J.},
   title={Viscosity effects in sound waves of finite amplitude},
   conference={
      title={Surveys in mechanics},
   },
   book={
      publisher={Cambridge, at the University Press, },
   },
   date={1956},
   pages={250--351 (2 plates)},
   review={\MR{0077346}},
}
\bib{mrrs1}{article}{
   author={Merle, Frank},
   author={Rapha\"el, Pierre},
   author={Rodnianski, Igor},
   author={Szeftel, Jeremie},
   title={On the implosion of a compressible fluid I: Smooth self-similar
   inviscid profiles},
   journal={Ann. of Math. (2)},
   volume={196},
   date={2022},
   number={2},
   pages={567--778},
   issn={0003-486X},
   review={\MR{4445442}},
   doi={10.4007/annals.2022.196.2.3},
}
\bib{mrrs2}{article}{
   author={Merle, Frank},
   author={Rapha\"el, Pierre},
   author={Rodnianski, Igor},
   author={Szeftel, Jeremie},
   title={On the implosion of a compressible fluid II: Singularity
   formation},
   journal={Ann. of Math. (2)},
   volume={196},
   date={2022},
   number={2},
   pages={779--889},
   issn={0003-486X},
   review={\MR{4445443}},
   doi={10.4007/annals.2022.196.2.4},
}
\bib{ps}{article}{
   author={Pol\'a\v cik, P.},
   author={\v Sver\'ak, V.},
   title={Zeros of complex caloric functions and singularities of complex
   viscous Burgers equation},
   journal={J. Reine Angew. Math.},
   volume={616},
   date={2008},
   pages={205--217},
   issn={0075-4102},
   review={\MR{2369491}},
   doi={10.1515/CRELLE.2008.022},
}
\bib{ser}{article}{
   author={Serre, Denis},
   title={Local existence for viscous system of conservation laws:
   $H^s$-data with $s>1+d/2$},
   conference={
      title={Nonlinear partial differential equations and hyperbolic wave
      phenomena},
   },
   book={
      series={Contemp. Math.},
      volume={526},
      publisher={Amer. Math. Soc., Providence, RI},
   },
   isbn={978-0-8218-4976-7},
   date={2010},
   pages={339--358},
   review={\MR{2731999}},
   doi={10.1090/conm/526/10388},
}
\bib{tao}{book}{
   author={Tao, Terence},
   title={Nonlinear dispersive equations},
   series={CBMS Regional Conference Series in Mathematics},
   volume={106},
   note={Local and global analysis},
   publisher={Conference Board of the Mathematical Sciences, Washington, DC;
   by the American Mathematical Society, Providence, RI},
   date={2006},
   pages={xvi+373},
   isbn={0-8218-4143-2},
   review={\MR{2233925}},
   doi={10.1090/cbms/106},
}
\bib{tay}{book}{
   author={Taylor, Michael E.},
   title={Partial differential equations III. Nonlinear equations},
   series={Applied Mathematical Sciences},
   volume={117},
   edition={3},
   publisher={Springer, Cham},
   date={[2023] \copyright 2023},
   pages={xxiii+755},
   isbn={978-3-031-33927-1},
   isbn={978-3-031-33928-8},
   review={\MR{4703941}},
   doi={10.1007/978-3-031-33928-8},
}
\bib{vh}{article}{
   author={Vol\cprime pert, A. I.},
   author={Hudjaev, S. I.},
   title={The Cauchy problem for composite systems of nonlinear differential
   equations},
   language={Russian},
   journal={Mat. Sb. (N.S.)},
   volume={87(129)},
   date={1972},
   pages={504--528},
   issn={0368-8666},
   review={\MR{0390528}},
}
\bib{wid}{book}{
   author={Widder, D. V.},
   title={The heat equation},
   series={Pure and Applied Mathematics},
   volume={Vol. 67},
   publisher={Academic Press [Harcourt Brace Jovanovich, Publishers], New
   York-London},
   date={1975},
   pages={xiv+267},
   review={\MR{0466967}},
}
\end{biblist}
\end{bibdiv}
\end{document}